\documentclass[a4paper,11pt,twoside]{article}
\usepackage{times}
\usepackage{makeidx}
\usepackage{amsfonts}
\usepackage{fullpage}
\usepackage{amsmath}
\usepackage{amssymb}
\usepackage{amsthm}
\usepackage{appendix}
\usepackage{manfnt}
\usepackage{marginnote}
\usepackage{xfrac}
\newtheorem{thm}{Theorem}[section]
\newtheorem{lem}[thm]{Lemma}
\newtheorem{prop}[thm]{Proposition}
\newtheorem{cor}[thm]{Corollary}
\theoremstyle{definition}
\newtheorem{ex}[thm]{Example}
\newtheorem{defn}[thm]{Definition}

\theoremstyle{remark}
\newtheorem{rmk}[thm]{Remark}
\usepackage{theoremref}
\usepackage[all]{xy}
\usepackage{mathrsfs}
\newcommand{\op}{\text{op}}

\newcommand{\fp}{\text{fp}}

\newcommand{\Ab}{\textbf{Ab}}
\newcommand{\Set}{\textbf{Set}}

\newcommand{\Ext}{\text{Ext}}

\newcommand{\rmods}{\textbf{mod}\text{-}}
\newcommand{\lmods}{\text{-}\textbf{mod}}
\newcommand{\Rmods}{\textbf{Mod}\text{-}}
\newcommand{\Lmods}{\text{-}\textbf{Mod}}

\begin{document}
\title{Duality and contravariant functors in the representation theory of artin algebras}
\author{Samuel Dean}
\maketitle

\abstract{We know that the model theory of modules leads to a way of obtaining definable categories of modules over a ring $R$ as the kernels of certain functors $(R\Lmods)^\op\to\Ab$ rather than of functors $R\Lmods\to\Ab$ which are given by a pp pair. This paper will give various algebraic characterisations of these functors in the case that $R$ is an artin algebra.

Suppose that $R$ is an artin algebra. An additive functor $G:(R\Lmods)^\op\to\Ab$ preserves inverse limits and $G|_{(R\lmods)^\op}:(R\lmods)^\op\to\Ab$ is finitely presented if and only if there is a sequence of natural transformations $(-,A)\to(-,B)\to G\to 0$ for some $A,B\in R\lmods$ which is exact when evaluated at any left $R$-module. Any additive functor $(R\Lmods)^\op\to\Ab$ with one of these equivalent properties has a definable kernel, and every definable subcategory of $R\Lmods$ can be obtained as the kernel of a family of such functors.

In the final section a generalised setting is introduced, so that our results apply to more categories than those of the form $R\Lmods$ for an artin algebra $R$. That is, our results are extended to those locally finitely presented $K$-linear categories whose finitely presented objects form a dualising variety, where $K$ is a commutative artinian ring.}
\tableofcontents
\section{Introduction}
The aim of this paper is to describe, in the context of artin algebras, the contravariant functors which arise from pp pairs in the model theory of modules. This provides an analogy with the relationship between pp pairs and the covariant functors which arise from them. We will also describe the relationship between the contravariant and covariant functors which arise from pp pairs.
\\\\\textbf{Acknowledgements.} The author would like to thank his PhD supervisor Mike Prest for valuable comments and motivation. He would also like to thank Lorna Gregory for useful discussions and encouragement (especially for pointing out other work on dualising varieties) and Jeremy Russell for stimulating conversations.
\subsection{Covariant functors which arise from pp pairs and definable subcategories}
We begin by recalling the notion of a pp pair and the relationship between pp pairs and covariant functors.

For a ring $R$, we write $R\Lmods$ for the category of left $R$-modules, $\Rmods R$ for the category of right $R$-modules, $R\lmods$ for the category of finitely presented left $R$-modules, and $\rmods R$ for the category of finitely presented right $R$-modules.

Let $R$ be a ring. A \textbf{pp formula (in the language of left $R$-modules)} is a condition $\theta(x_1,\dots,x_n)$ in the free variables $x_i$ of the form
\begin{displaymath}
\exists y_1,\dots y_m A\left(\begin{matrix}
x_1\\\vdots\\x_n\end{matrix}\right)=B\left(\begin{matrix}
y_1\\\vdots\\y_m\end{matrix}\right)
\end{displaymath}
where $A$ and $B$ are appropriately sized matrices with entries from $R$. For such a formula and any left $R$-module $M$, we can take the \textbf{solution set}
\begin{displaymath}
\theta(M)=\left\{(a_1,\dots,a_n)\in M^n:\exists b_1,\dots, b_m\in M\text{ such that }A\left(\begin{matrix}
a_1\\\vdots\\a_n\end{matrix}\right)=B\left(\begin{matrix}
b_1\\\vdots\\b_m\end{matrix}\right)\right\}\subseteq M^n,
\end{displaymath}
which is actually a subgroup of $M^n$. A \textbf{pp pair} $\varphi/\psi$ consists of pp formulas $\varphi$ and $\psi$, in the same free variables, such that $\psi(M)\subseteq \varphi(M)$ for any left $R$-module $M$ (although it is enough to check this condition on finitely presented modules \cite[1.2.23]{prest2009}).

For a pp pair $\varphi/\psi$ in the language of left $R$-modules, we denote the corresponding functor by $\mathbb{F}_{\varphi/\psi}:R\Lmods\to\Ab$, which takes a left $R$-module $M$ to the quotient group $\varphi(M)/\psi(M)$, where $\varphi(M)$ and $\psi(M)$ are, respectively, the solution sets of $\varphi$ and $\psi$ in $M$. The solution sets $\varphi(M)$ and $\psi(M)$ are subgroups of $M^n$ for some positive integer $n$, and the solutions pp formulas are preserved by homomorphisms, so $\mathbb{F}_{\varphi/\psi}$ is defined on morphisms in the obvious way.

\begin{thm}\thlabel{covarconditions}(From \cite[Subsection 10.2.8]{prest2009}) For an additive functor $F:R\Lmods\to\Ab$, the following are equivalent.
\begin{list}{*}{}
\item[(a)]$F$ preserves direct limits and is finitely presented when restricted to $R\lmods$.
\item[(b)]There is a sequence of natural transformations 
\begin{displaymath}
\xymatrix{(B,-)\ar[r]&(A,-)\ar[r]&F\ar[r]&0}
\end{displaymath}
which is exact when evaluated at any left $R$-module, where $A$ and $B$ are finitely presented left $R$-modules.
\item[(c)]$F\cong \mathbb{F}_{\varphi/\psi}$ for some pp pair $\varphi/\psi$.
\end{list}
\end{thm}

A \textbf{definable subcategory} of a module category $R\Lmods$ is the kernel of a family of additive functors $R\Lmods\to\Ab$, each of which is given by a pp pair.

Part of the importance of definable subcategories of $R\Lmods$ is that they correspond to the closed subsets of a topological space which is an invariant of the category $R\Lmods$. The \textbf{Ziegler spectrum} of $R$ is a topological space $\text{Zg}_R$ whose set of points contains of element of each isomorphism class of the set of all indecomposable pure-injective left $R$-modules. The points of $\text{Zg}_R$ form a set (as opposed to a proper class, see \cite[4.3.38]{prest2009}). The closed sets of $\text{Zg}_R$ are those of the form $\mathcal{D}\cap\text{Zg}_R$ for a definable subcategory $\mathcal{D}\subseteq R\Lmods$. A definable subcategory $\mathcal{D}$ can be recovered from the set $\mathcal{D}\cap\text{Zg}_R$ (see \cite[5.1.4]{prest2009}).

\thref{covarconditions} holds when $R$ is any ring, but suppose that $R$ is an artin algebra. Under this assumption, we will prove an analogous result which gives necessary and sufficient conditions for a functor $(R\Lmods)^\op\to\Ab$ to be given by a pp pair. Furthermore, we will show that any functor $(R\Lmods)^\op\to\Ab$ which satisfies these conditions has a definable kernel, and that every definable subcategory of $R\Lmods$ can be obtained as the kernel of a family of such functors. Therefore, we have a new way of confirming definability of subcategories in this context. The algebraic characterisation of these contravariant functors allows for more streamlined proofs of the definability of certain subcategories (e.g. the argument for the definability of the subcategory of modules over a tubular algebra which are of a fixed slope in \cite{harland2010}).
\subsection{Dualisation and predualisation of functors}\label{defs}
A key ingredient in this paper is dualisation, so we now introduce that. Throughout this paper $K$ will denote a commutative artinian ring unless stated otherwise.

Let $J$ be the $K$-module
\begin{displaymath}
J=\bigoplus^n_{i=1}E(S_i)
\end{displaymath}
where $\{S_i:i=1,\dots,n\}$ is a minimal set of representatives of the simple left $K$-modules and $E(S_i)$ denotes the injective hull of $S_i$ for each $i=1,\dots,n$. In particular, if $K$ is a field then $J\cong K$. Note that $J$ is an injective cogenerator of $K\Lmods$ (see, for example, \cite[18.2, 18.15]{andersonfuller}). 

For a $K$-module $M$ we define its \textbf{dual} to be the $K$-module $M^*=\text{Hom}_K(M,J)$. For a linear map $f:M\to N$ between $K$-modules $M$ and $N$, we define its \textbf{dual} to be the induced linear map $f^*:N^*\to M^*:t\mapsto tf$. 

For any $K$-module $M$, there is a canononical $K$-linear map $\eta_M:M\to M^{**}$ which sends an element $x\in M$ to the linear map $\eta_M(x)\in M^{**}$ given by $\eta_M(x)(f)=f(x)$ for all $f\in M^*$.

It is well known (see \cite[II.3]{ars}) that, for any finitely generated $K$-module $M$, $M^*$ is finitely generated and the canonical map $\eta_M:M\to M^{**}$ is an isomorphism.

For a $K$-linear category $A$, consider an additive functor $F:A\to \Ab$. It has a unique factorisation $F=UF'$ where $F':A\to K\Lmods$ is $K$-linear and $U:K\Lmods\to\Ab$ is the forgetful functor. That is, for every object $a\in A$, $Fa$ has a hidden $K$-module structure, and every morphism in $A$ is sent to a $K$-linear map.\footnote{In order to define the $K$-module structure on $Fa$ for an object $a\in A$, consider that each element $k\in K$ gives an endomorphism $k1_a:a\to a$, and hence a homomorphism $F(k1_a):Fa\to Fa$.} Therefore, we can make the following definition.
\begin{defn}Let $A$ be a $K$-linear category and let $F:A\to\Ab$ be an additive functor. The \textbf{dual} of $F$ is the functor $F^*:A^\op\to\Ab$ which sends an object $a\in A$ to $(Fa)^*$ and sends a morphism $(f:a\to b)\in A$ to the induced map
\begin{displaymath}
(Ff)^*:(Fb)^*\to (Fa)^*.
\end{displaymath}
\end{defn}

\begin{ex}If $A=R$, a $K$-algebra, in the previous definition, we recover the usual process of obtaining a right $R$-module $M^*$ from a left $R$-module $M$. 
\end{ex}

Since $J$ is a cogenerator of $K\Lmods$, $M^*=0$ iff $M=0$ for any $K$-module $M$. This fact is crucial to later observations.

If $R$ is an artin algebra over $K$ then a left or right $R$-module is finitely generated iff it is finitely presented iff it is finitely generated as a $K$-module. It follows that, for any finitely presented (left or right) $R$-module $M$, $M^*$ is also finitely presented and the canonical map $\eta_M:M\to M^{**}$ is an isomorphism of $R$-modules. Therefore, as is well known, there is an equivalence $$(-)^*:(R\lmods)^\op\to \rmods R$$
with inverse $$(-)^*:\rmods R\to(R\lmods)^\op.$$
For a functor $F:R\Lmods\to\Ab$ we define its \textbf{predual} $F_*:(\Rmods R)^\op\to\Ab$ to be the functor given by $F_*(M)=F(M^*)$. Similar definitions can be made for a functor $R\lmods\to\Ab$, $\Rmods R\to\Ab$ or $(R\Lmods)^\op\to\Ab$.

Frequent use will be made of the following result.
\begin{lem}\thlabel{homtens}For any left $R$-module $M$ and right $R$-module $N$, there is are isomorphisms
\begin{displaymath}
(M,N^*)\cong(N\otimes M)^*\cong(N,M^*)
\end{displaymath}
which are natural in $M$ and $N$.
\end{lem}
\begin{proof}
Both isomorphisms are instances of the hom-tensor adjunction (see, for example, \cite[19.11]{andersonfuller}). Alternatively, one can show that $(M,N^*)\cong (N,M^*)$ directly, which gives an adjunction $\Phi\dashv\Psi$ where the functors $\Phi:R\Lmods\to(\Rmods R)^\op$ and $\Psi:(\Rmods R)^\op\to R\Lmods$ are given by dualisation (the unit of adjunction at $M\in R\Lmods$ is the canonical map $M\to M^{**}$).
\end{proof}
\begin{rmk}[A note for those who know about monads]\thlabel{monadic}The adjunction $\Phi\dashv\Psi$ referred to in the proof of \thref{homtens} can be shown to be weakly monadic or ``weakly tripleable"(see \cite[V.7 Exercises]{maclane1998} and combine with \cite[18.1(d), 18.14(d)]{andersonfuller}). Thus double-dualisation gives a monad on $R\Lmods$ whose category of algebras is $(\Rmods R)^\op$. This mitigates the absence of an equivalence $R\Lmods\simeq (\Rmods R)^\op$, since dualisation gives a monadic functor $(-)^*:(\Rmods R)^\op\to R\Lmods$.
\end{rmk}
Some of the results in this paper require the properties of tensor products of functors. We briefly summarise these.
\begin{defn}For a small pre-additive category $A$, and functors $M:A\to\Ab$ and $N:A^\op\to\Ab$, their \textbf{tensor product} is given by the coend formula (see \cite[IX.6]{maclane1998} for ends and coends)
\begin{displaymath}
N\otimes_A M=\int^{a\in A}(Na)\otimes_\mathbb{Z}(Ma)\in\Ab.
\end{displaymath}
\end{defn}
\begin{rmk}If $A$ is a $K$-linear category, then for any $M\in A\Lmods$ and $N\in\Rmods A$, $N\otimes_A M$ is a $K$-module, and coincides with the coend $\int^{a\in A}(Na)\otimes_K(Ma)\in K\Lmods$.
\end{rmk}
This definition of tensor products is similar to that given in \cite[IX.6]{maclane1998} for $\Set$-valued functors, and coincides with the definition given in \cite{fisher1968}.
\begin{lem}\thlabel{tensyoneda}Let $A$ be a small pre-additive category. For any additive functor $F:A\to\Ab$ and any object $a\in A$, there is an isomorphism
\begin{displaymath}
A(-,a)\otimes _A F\cong Fa
\end{displaymath}
which is natural in $F$ and $a$.
\end{lem}
\begin{proof}
See either \cite[Proposition 1.1]{fisher1968} or use universal properties of coends to verify that the morphisms $A(b,a)\otimes_\mathbb{Z} Fb\to Fa:f\otimes x\mapsto (Ff)x$ assemble to form an isomorphism $A(-,a)\otimes_A F\to Fa$.
\end{proof}
\begin{lem}\thlabel{homtenspreadd}Let $A$ be a $K$-linear category and let $M\in A\Lmods$ and $N\in\Rmods A$ be given. There are isomorphisms
\begin{displaymath}
(M,N^*)\cong (N\otimes_A M)^*\cong (N,M^*)
\end{displaymath}
which are natural in $M$ and $N$.
\end{lem}
\begin{proof}
See \cite[Corollary 2.4]{fisher1968}. One can also construct a proof by combining some well-known facts about the calculus of ends and coends \cite[IX]{maclane1998} with \thref{homtens}.
\end{proof}
\subsection{The annihilator functor of a pp pair}
Let $R$ be an artin algebra.

For a pp formula $\theta(x_1,\dots, x_n)$ in the language of left $R$-modules, we will write $\mathbb{F}_\theta:R\Lmods\to\Ab$ for the corresponding functor, and $F_\theta=\mathbb{F}_\theta|_{R\lmods}$. We also define the functor $\mathbb{A}_\theta:(R\Lmods)^\op\to\Ab$ at a left $R$-module $M$ by \begin{displaymath}
\mathbb{A}_\theta(M)=\left\{\overline f\in (M^*)^n:\overline f\cdot \overline a=0\text{ for all }\overline a\in\theta(M)\right\},
\end{displaymath}
where for $\overline f=(f_1,\dots,f_n)\in (M^*)^n$ and $\overline a=(a_1,\dots,a_n)\in M^n$ we write $\overline f\cdot\overline a$ to denote the sum $\sum^n_{i=1}f_i(a_i)$. Furthermore, we define $A_\theta$ to be the restriction of $\mathbb{A}_\theta$ to $(R\lmods)^\op$.

For a pp pair $\varphi/\psi$ in the language of left $R$-modules, we will write $\mathbb{F}_{\varphi/\psi}=\mathbb{F}_\varphi/\mathbb{F}_\psi:R\Lmods\to\Ab$ for the corresponding functor, and $F_{\varphi/\psi}=\mathbb{F}_{\varphi/\psi}|_{R\lmods}$. Note that $\mathbb{A}_\varphi$ is a subfunctor of $\mathbb{A}_\psi$, and so we may define $\mathbb{A}_{\varphi/\psi}=\mathbb{A}_\psi/\mathbb{A}_\varphi$. We also define $A_{\varphi/\psi}$ to be the restriction of $\mathbb{A}_{\varphi/\psi}$ to $(R\lmods)^\op$. See Section \ref{annpp} for the relationship between these annihilator functors and Prest's notion of duality for pp formulas. We call $\mathbb{A}_{\varphi/\psi}$ the \textbf{annihilator functor} of $\varphi/\psi$.

It is known that a full subcategory $\mathcal{D}$ of $R\Lmods$ is definable iff there is some set $T$ of pp pairs in the language of left $R$-modules such that 
\begin{displaymath}
M\in\mathcal{D}\Leftrightarrow \forall \varphi/\psi\in T, \mathbb{A}_{\varphi/\psi}M=0
\end{displaymath}
for all $M\in R\Lmods$. This is due to the fact that, for any pp pair $\varphi/\psi$ and any left $R$-module $M$, $\mathbb{A}_{\varphi/\psi}M=0$ iff $\mathbb{F}_{\varphi/\psi}M=0$ (see \cite[1.3.15]{prest2009}). This motivates the desire to find algebraic characterisations of annihilator functors.
\section{Contravariant functors}
\subsection{The main theorems}
Now we have enough definitions to state the two main theorems. In this section $R$ will denote an artin $K$-algebra.
\begin{thm}\thlabel{contravarconditions}Let $G:(R\Lmods)^\op\to\Ab$ be an additive contravariant functor. The following are equivalent:
\begin{list}{*}{â€¢}
\item[(a)]$G\cong F^*$ for some functor $F:R\Lmods\to\Ab$ which preserves direct limits and is finitely presented when restricted to $R\lmods$.
\item[(b)]$G$ preserves inverse limits and is finitely presented when restricted to $(R\lmods)^\op$.
\item[(c)]There are finitely presented left $R$-modules $A$ and $B$ and a sequence of natural transformations
\begin{displaymath}
\xymatrix{(-,A)\ar[r]&(-,B)\ar[r]&G\ar[r]&0}
\end{displaymath}
which is exact when evaluated at any left $R$-module.
\item[(d)]There is a sequence of natural transformations
\begin{displaymath}
\xymatrix{G'\ar[r]&G''\ar[r]&G\ar[r]&0}
\end{displaymath}
which is exact when evaluated at any left $R$-module, where the functors $G',G'':(R\Lmods)^\op\to\Ab$ preserve inverse limits and are finitely presented when restricted to $(R\lmods)^\op$.
\item[(e)]$G\cong F_*$ for some functor $F:\Rmods R\to\Ab$ which preserves direct limits and is finitely presented when restricted to $\rmods R$.
\item[(f)]$G\cong\mathbb{A}_{\varphi/\psi}$ for some pp pair $\varphi/\psi$ in the language of left $R$-modules.
\end{list}
\end{thm}
\begin{thm}\thlabel{pppairdual}For any pp pair $\varphi/\psi$ in the language of left $R$-modules, $\mathbb{F}_{\varphi/\psi}^*\cong\mathbb{A}_{\varphi/\psi}$.
\end{thm}
The proofs of these theorems are long because they have many different parts. To accommodate them, the proofs are contained Section \ref{proof} along with some necessary background. To help the reader navigate Section \ref{proof}, the locations of key points of the proofs are given by the table below.
\begin{center}
\begin{tabular}{c|c}
Key point & Location \\
\hline 
(a) implies (b)&\thref{aimpb2}\\
(b) implies (a)&\thref{bimpa}\\
(c) implies (d)&\thref{bimpa}\\
(d) implies (b)&\thref{dimpb}\\
(c) implies (e)&\thref{cimpe}\\
(e) implies (c)&\thref{cimpe}\\
(e) is equivalent to (f)&\thref{eeqf}\\
\thref{pppairdual}&\thref{meh}
\end{tabular}
\end{center}
\subsection{Examples of contravariant functors with definable kernels}\label{appa}
Let $\mathcal{D}$ be a definable subcategory of $R\Lmods$. Then
\begin{displaymath}
\mathcal{D}=\{M\in R\Lmods:FM=0\text{ for all }F\in S\}
\end{displaymath}
where $S$ is a set of functors $R\Lmods\to\Ab$ each satisfying one of the equivalent conditions of \thref{covarconditions}. Then, for each $F\in S$ and each left $R$-module $M$, $\overrightarrow{F}^*(M)=0$ iff $(\overrightarrow FM)^*=0$ iff $\overrightarrow{F}M=0$ iff $M\in\mathcal{D}$. Therefore, we can express $\mathcal{D}$ as the kernel of a set of contravariant functors $(R\Lmods)^\op\to\Ab$, each of which satisfies one of the equivalent conditions listed in \thref{contravarconditions}. (Of course, the kernel of a family of functors $(R\Lmods)^\op\to\Ab$ is, strictly speaking, a subcategory of $(R\Lmods)^\op$, but here we refer to the corresponding subcategory of $R\Lmods$.) Conversely, let $G:(R\Lmods)^\op\to\Ab$ be a functor satisfying any of the equivalent conditions in \thref{contravarconditions}. Then $G\cong F^*$ for some functor $F$ which satisfies any of the equivalent conditions in \thref{covarconditions}. For any left $R$-module $M$, $GM=0$ iff $FM=0$. Therefore the kernel of $G$ is definable. The family of definable subcategories of $R\Lmods$ is closed under intersection, so the kernel of any family of such functors $(R\Lmods)^\op\to\Ab$ is definable.
\begin{ex}The most basic example is that of a hom-functor $(-,A):(R\Lmods)^\op\to\Ab$ for $A\in R\lmods$. Of course it satisfies the equivalent conditions in \thref{contravarconditions}, and so has a definable kernel, but this is also seen by using the hom-tensor duality (see \thref{homtens}) to obtain the natural isomorphisms\begin{displaymath}
(-,A)\cong(-,A^{**})\cong (A^*\otimes-)^*,
\end{displaymath}
so $(-,A)$ has a definable kernel because $A^*\otimes -$ does (it satisfies one of the equivalent conditions in \thref{covarconditions} by \cite[6.1]{auslander1965}).
\end{ex}
\begin{ex}In this example it is shown that for any finitely presented left $R$-module $A$, $\Ext^n(-,A):(R\Lmods)^\op\to\Ab$ satisfies the equivalent conditions in \thref{contravarconditions} for all $n>0$.

There is an exact sequence
\begin{displaymath}
\xymatrix{0\ar[r]&L\ar[r]&P\ar[r]&A^*\ar[r]&0}
\end{displaymath}
where $P$ is a finitely generated projective right $R$-module and $L\in \rmods R$. Therefore there is an exact sequence
\begin{displaymath}
\xymatrix{0\ar[r]&A\ar[r]&P^*\ar[r]&L^*\ar[r]&0.}
\end{displaymath}
Note that $P^*$ is injective in $R\Lmods$ since $(-,P^*)\cong (P,(-)^*):(R\Lmods)^\op\to\Ab$ is exact because $P$ is projective in $\Rmods R$. This induces a sequence of natural transformations
\begin{displaymath}
\xymatrix{0\ar[r]&(-,A)\ar[r]&(-,P^*)\ar[r]&(-,L^*)\ar[dl]&&&&\\
&&\Ext^1(-,A)\ar[r]&\Ext^1(-,P^*)(=0)\ar[r]&\Ext^1(-,L^*)\ar[dl]\\
&&&\Ext^2(-,A)\ar[r]&\Ext^2(-,P^*)(=0)\ar[r]&\dots.}
\end{displaymath}
which is exact when evaluated at any left $R$-module. 

Therefore, $\Ext^1(-,A):(R\Lmods)^\op\to\Ab$ satisfies condition (c) in \thref{contravarconditions}. Also, for any $n>0$, $\Ext^{n+1}(-,A)\cong \Ext^n(-,L^*)$.

Since $L^*$ is also a finitely presented left $R$-module, it follows by induction that for any $n>0$, $\Ext^n(-,A):(R\Lmods)^\op\to\Ab$ satisfies condition (c) in \thref{contravarconditions}.
\end{ex}
\begin{ex}For any morphism $(f:A\to B)\in R\lmods$, let $G_f:(R\Lmods)^\op\to\Ab$ be the functor defined by the exactness of the sequence
\begin{displaymath}
\xymatrix{(-,A)\ar[r]^{(-,f)}&(-,B)\ar[r]&G_f\ar[r]&0}
\end{displaymath}
when evaluated at any left $R$-module. For any left $R$-module $M$, $G_fM=0$ iff $M$ is projective over $f$.

Since $G_f$ satisfies one of the equivalent conditions listed in \thref{contravarconditions}, its kernel, gives a definable subcategory of $R\Lmods$. 

Furthermore, the subcategory $\mathcal{D}$ of $R\Lmods$ which consists of the left $R$-modules which are projective over every epimorphism in $R\lmods$ is a definable subcategory of $R\Lmods$, since $\mathcal{D}=\{M\in R\Lmods:G_fM=0\text{ for all epimorphisms }f\in R\lmods\}$. However, this is not a new definable subcategory: It is the category of flat left $R$-modules, which is known to be definable because $R$ is right coherent (see \cite[3.4.24]{prest2009}). To see that $\mathcal{D}$ is the category of flat modules, note that it contains all flat modules since it contains all projective modules and is closed under direct limits. By \thref{homtens}, for any left $R$-module $M\in\mathcal{D}$, $(M,-):R\lmods\to\Ab$ is exact, which implies that $(-\otimes M)^*\cong (M,(-)^*):(R\lmods)^\op\to \Ab$ is exact and therefore $-\otimes M:R\lmods\to\Ab$ is exact. Therefore $\text{Tor}_1(-,M)|_{\rmods R}=0$, which is to say that $M$ is flat.
\end{ex}
\subsection{Example: The contravariant projectively stable hom-functor}
The Auslander-Reiten duality is an important tool in the representation theory of artin algebras, and is used to prove the existence of almost split sequences (see, for example, \cite[IV.3]{ass}). Sometimes it is only given for finitely presented modules (i.e. the way it is presented at \cite[IV.2.13]{ass}), but one can show that it holds when one variable is not necessarily finitely presented (i.e. as given by Krause in \cite{krause2002}). Let $R$ be an artin algebra throughout this section. As a longer example, we will use \thref{contravarconditions} to show how one can extend the Auslander-Reiten duality for finitely presented $R$-modules so that it holds with an arbitrarily large module in one variable.

For a left $R$-module $M$, let $M^t=\text{Hom}_R(M,R)$, which is a right $R$-module with $(fr)(x)=f(x)r$ for all $f\in M^t$, $x\in M$ and $r\in R$. For each $N\in R\Lmods$, we have a functor $(-)^t\otimes N:(R\Lmods)^\op\to\Ab$. 
\begin{prop}\thlabel{tix}If $N$ is a finitely presented left $R$-module, then the functor $(-)^t\otimes N:(R\Lmods)^\op\to\Ab$ preserves inverse limits and is finitely presented when restricted to $(R\lmods)^\op$.
\end{prop}
\begin{proof}
Find positive integers $m$ and $n$ and a right exact sequence $R^m\to R^n\to N\to 0$. Then we have a sequence of natural transformations $(-)^t\otimes R^m\to (-)^t\otimes R^n\to(-)^t\otimes  N\to 0$ which is exact when evaluated at any left $R$-module. Therefore, by \thref{contravarconditions}, we need only prove that $(-)^t\otimes R$ satisfies these criteria, and indeed it does since $(-)^t\otimes R\cong (-,R)$.
\end{proof}
\begin{defn}For left $R$-modules $M$ and $N$, write $\underline{\text{Hom}}_R(M,N)=(M,N)/P(M,N)$, where $P(M,N)$ denotes the group of maps $M\to N$ which factor through a projective left $R$-module. 
\end{defn}
\begin{thm}\thlabel{stabhom}For any finitely presented left $R$-module $N$, the functor $\underline{\text{Hom}}_R(-,N):(R\Lmods)^\op\to\Ab$ preserves inverse limits and is finitely presented when restricted to $(R\lmods)^\op$.
\end{thm}
\begin{proof}
By the same proof as that of \cite[IV.2.12]{ass}, there is a sequence of natural transformations
\begin{displaymath}
\xymatrix{(-)^t\otimes N\ar[r]&(-,N)\ar[r]&\underline{\text{Hom}}_R(-,N)\ar[r]&0}
\end{displaymath}
which is exact when evaluated at any left $R$-module. By \thref{contravarconditions}(d) and \thref{tix}, this finishes the proof.
\end{proof}
\begin{cor}[See \cite{krause2002} for an alternative proof.]\thlabel{arform}Let $N$ be a finitely presented left $R$-module. For any left $R$-module $M$ there is an isomorphism
\begin{displaymath}
\text{Ext}^1(N,M)^*\cong\underline{Hom}_R(M,\tau N)
\end{displaymath}
which is natural in $M$ and $N$, where $\tau$ is the Auslander-Reiten translation.
\end{cor}
\begin{proof}This proof relies on a fact from Section \ref{proof}, so the reader may wish to return to it later. By the Auslander-Reiten formula \cite[IV.2.13]{ass}, there is such a natural isomorphism for $M\in R\lmods$. However, both $\text{Ext}^1(N,-)^*$ and $\underline{\text{Hom}}_R(-,\tau N)$ preserve inverse limits (by \cite[Corollary of Theorem 1]{brown1975} and \thref{stabhom}), so by \thref{unqext} they are, up to isomorphism, the same functor $(R\Lmods)^\op\to\Ab$.
\end{proof}
\section{Background and proof of the main theorems}\label{proof}
The main results are proven by a number of steps which require some background knowledge. We will go through the background in this section, and build up to proving the theorems. Throughout, let $R$ denote an artin algebra.
\subsection{Finitely presented objects and locally finitely presented categories}\label{lfp}
This section is just background on finitely presented objects and finitely presented functors.

An object $c\in C$ in an category $C$ is said to be \textbf{finitely presented} if $C(c,-):C\to\Ab$ preserves direct limits. Note that this does agree with the usual notions of finitely presented module and finitely presented functor, as stated in \thref{fpfunc1}, \thref{fpfunc2}, and \thref{fpfunc3}.

\begin{defn}\cite{crawleyboevey1994} We say that an additive category $C$ is \textbf{locally finitely presented} if it has direct limits, every object of $C$ is the direct limit of a direct system of finitely presented objects, and $C^\text{fp}$, the full subcategory of finitely presented objects in $C$, is skeletally small. (Note that we do not insist that $C$ is complete or cocomplete, so this is weaker than the notion of ``locally finitely presented'' in \cite{adamekrosicky}.)
\end{defn}
For any small pre-additive category $A$, we write $A\Lmods$ for the category of all additive functors $A\to\Ab$, and we define $A\lmods=(A\Lmods)^\fp$. We write $\Rmods A=A^\op\Lmods$ and $\rmods A=A^\op\lmods$.
\begin{thm}\cite[Section 1.2]{crawleyboevey1994}\thlabel{fpfunc1} Suppose $A$ is a small additive category. Then $A\Lmods$ is locally finitely presented, and $A\lmods$ consists of those functors $F$ which appear in an exact sequence of the form $A(a,-)\to A(b,-)\to F\to 0$ for some $a,b\in A$.
\end{thm}
\begin{cor}\thlabel{fpfunc2}If $A$ is any small pre-additive category then $A\Lmods$ is locally finitely presented, and $A\lmods$ consists of those functors $F$ which appear in an exact sequence of the form $\bigoplus^m_{i=1}A(a_i,-)\to \bigoplus^n_{j=1}A(b_j,-)\to F\to 0$ for positive integers $m$ and $n$ and objects $a_1,\dots,a_m,b_1,\dots,b_n\in A$.
\end{cor}
\begin{proof}
One can construct a small additive category $A^+$ which is freely generated by $A$: its objects are finite strings of objects from $A$ and its morphisms are matrices consisting of morphisms of $A$. The inclusion functor $A\to A^+$ induces an equivalence $A^+\Lmods\to A\Lmods$. Therefore $A\Lmods$ is locally finitely presented and the finitely presented objects can therefore be found by applying this equivalence to those in $A^+\Lmods$, hence the stated characterisation.
\end{proof}
\begin{cor}\thlabel{fpfunc3}If $A$ is a ring then $A\Lmods$ is locally finitely presented, and $A\lmods$ consists of those modules $M$ which appear in an exact sequence of the form $A^m\to A^n\to M\to 0$ where $m$ and $n$ are positive integers.
\end{cor}
\begin{thm}\thlabel{auslander}\cite[2.1]{auslander1965} Let $A$ be an abelian category and let $P$ be a full additive subcategory of $A$ consisting only of projective objects. Let $P'$ be the full subcategory of $A$ consisting of all objects $a\in A$ for which there is an exact sequence $p\to q\to a\to 0$ with $p,q\in P$. Then the following hold.
\begin{list}{*}{â€¢}
\item[(a)]$P'$ is closed under extensions.
\item[(b)]$P'$ is closed under cokernels.
\item[(c)]If $P$ has the additional property that, for any morphism $(q\to r)\in P$ there is an exact sequence $(p\to q\to r)\in A$ with $p\in P$, then $P'$ is closed under kernels.
\end{list}
\end{thm}
\begin{defn}Let $A$ be a pre-additive category. A \textbf{weak cokernel} of a morphism $(f:a\to b)\in A$ is a morphism $(g:b\to c)\in A$ such that, for every object $x\in A$ the sequence
\begin{displaymath}
\xymatrix{A(c,x)\ar[r]^{A(g,x)}&A(b,x)\ar[r]^{A(f,x)}&A(a,x)}
\end{displaymath}
is exact. If every morphism in $A$ has a weak cokernel then we say that $A$ \textbf{has weak cokernels} or is \textbf{with weak cokernels}.
\end{defn}
\begin{cor}\thlabel{weakcokernels}For any small additive category $A$ with weak cokernels, $A\lmods$ is an abelian subcategory of $A\Lmods$.
\end{cor}
\begin{cor}\thlabel{covarabelian}For any small pre-additive category $A$, $(A\lmods)\lmods$ is an abelian subcategory of $(A\lmods)\Lmods$.
\end{cor}
\begin{cor}\thlabel{noetheriancor}For any left coherent ring $R$, $R\lmods$ is an abelian subcategory of $R\Lmods$, and hence $\rmods(R\lmods)$ is an abelian subcategory of $\Rmods(R\lmods)$.
\end{cor}
\begin{rmk}In particular, \thref{noetheriancor} applies to artin algebras, since all artin algebras are coherent.
\end{rmk}\begin{lem}\thlabel{lemmaf}Let $C$ be a locally finitely presented additive category. Then the following hold.
\begin{list}{*}{}
\item[(a)]If $\alpha:F\to F'$ is a natural transformation between additive functors $F,F':C\to\Ab$ which preserve direct limits such that for all $c\in C^\fp$ the morphism $\alpha_c:Fc\to F'c$ is a monomorphism (respectively, an epimorphism) then for any $d\in C$ the morphism $\alpha_d:Fd\to F'd$ is a monomorphism (respectively, an epimorphism).
\item[(b)]If $\beta:G\to G'$ is a natural transformation between additive functors $G,G':C^\op\to\Ab$ which preserve inverse limits such that for all $c\in C^\fp$ the morphism $\beta_c:Gc\to G'c$ has the property of being either a monomorphism (respectively, an isomorphism) then for any $d\in C$ the morphism $\alpha_d:Fd\to F'd$ is monomorphism (respectively, an isomorphism).
\end{list}
\end{lem}
\begin{proof}
This follows directly from the fact that every object of $C$ is a direct limit of finitely presented objects and the fact that direct limits are exact and inverse limits are left exact and preserve isomorphisms (the same would hold if $\Ab$ was replaced by any abelian category with exact direct limits).
\end{proof}
\begin{lem}\thlabel{faithext}Let $F:C\to D$ be an additive functor from a locally finitely presented additive category $C$ to an additive category $D$ which has direct limits, and suppose $F$ preserves direct limits and sends finitely presented objects to finitely presented objects. If $F|_{C^\fp}:C^\fp\to D$ is (fully) faithful then $F:C\to D$ is (fully) faithful.
\end{lem}
\begin{proof}
Consider, for objects $a,b\in C$, the homomorphism
\begin{displaymath}
F_{ab}:C(a,b)\to D(Fa,Fb):f\mapsto Ff.
\end{displaymath}
It is natural in $a$ and $b$, and when $a$ is finitely presented, the functors $C(a,-),D(Fa,F-):C\to\Ab$ preserve direct limits since $D(Fa,F-)$ is the composition of functors
\begin{displaymath}
\xymatrix{C\ar[r]^F&D\ar[r]^{D(Fa,-)}&\Ab,}
\end{displaymath}
both of which preserve direct limits.
Therefore, since, when $a$ and $b$ are finitely presented the map $F_{ab}$ is a monomorphism, it follows that $F_{ab}$ is a monomorphism if $a$ is finitely presented by \thref{lemmaf}.

We have just shown that, for any object $b\in C$, the morphism $F_{ab}$ is a monomorphism when $a$ is finitely presented. The functors $C(-,b),D(F-,Fb):C^\op\to\Ab$ preserve inverse limits, since $D(F-,Fb)$ is the composition
\begin{displaymath}
\xymatrix{C^\op\ar[r]^{F^\op}&D^\op\ar[r]^{D(-,Fb)}&\Ab}
\end{displaymath}
of functors which preserve inverse limits.

It follows that $F_{ab}$ is a monomorphism for any objects $a,b\in C$ by \thref{lemmaf}.

One can replace ``monomorphism'' by ``isomorphism'' throughout the proof.
\end{proof}
\begin{defn}\thlabel{extdef}Let $C$ be a locally finitely presented functor and let $F:C^\fp\to\Ab$ be an additive functor. We define the functor $\overrightarrow{F}:C\to\Ab$ by $\overrightarrow{F}c=C(-,c)|_{C^{\fp\op}}\otimes _{C^\fp}F$.
\end{defn}
\begin{lem}Let $C$ be a locally finitely presented category and let $F:C^\text{fp}\to\Ab$ be an additive functor. There is an isomorphism $\overrightarrow{F}|_{C^\fp}\cong F$ which is natural in $F$ and $\overrightarrow{F}$ preserves direct limits.
\end{lem}
\begin{proof}
The first statement follows directly from the definition and \thref{tensyoneda}. That $\overrightarrow{F}$ preserves direct limits is due to the fact that it is the composition of functors
\begin{displaymath}
\xymatrix{C\ar[rr]^\Phi&&\Rmods C^\fp\ar[rr]^{-\otimes_{C^\fp}F}&&\Ab,}
\end{displaymath}
both of which preserve direct limits, where $\Phi$ sends $c\in C$ to $C(-,c)|_{C^{\fp\op}}$.
\end{proof}
\begin{lem}\thlabel{blah}Let $C$ be a locally finitely presented category and let $F:C\to\Ab$ be an additive functor which preserves direct limits. For any $c\in C$, there is an isomorphism
\begin{displaymath}
\overrightarrow{F|_{C^\fp}}c\to Fc
\end{displaymath}
which is natural in $F$ and $c$.
\end{lem}
\begin{proof}
For any $d\in C^\fp$, there is a morphism $C(d,c)\otimes_\mathbb{Z}Fd\to Fc:f\otimes x\mapsto (Fd)x$. By the universal property of tensor products, these assemble into a morphism $C(-,c)|_{C^{\fp\op}}\otimes_{C^\fp}F|_{C^\fp}\to Fc$ which is natural in $F$ and $c$. By the proof of \thref{homtens}, this is an isomorphism when $c$ is finitely presented. Since $F$ preserves direct limits, it follows from \thref{lemmaf} that this natural transformation is an isomorphism for any $c\in C$.
\end{proof}
\begin{thm}\thlabel{genunqext}Let $C$ be a locally finitely presented category and let $F:C^\text{fp}\to\Ab$ be an additive functor. Then $\overrightarrow{F}:C\to\Ab$ is the unique additive functor which preserves direct limits and $\overrightarrow{F}|_{C^\fp}\cong F$.
\end{thm}
\begin{proof}Suppose $F':C\to\Ab$ is any other such functor. Then since $F'$ preserves direct limits, by \thref{blah} there are isomorphisms
\begin{displaymath}
F'\cong \overrightarrow{F'|_{C^\fp}}\cong \overrightarrow{F},
\end{displaymath}
as required.
\end{proof}
\begin{lem}\thlabel{extex}Let $C$ be a locally finitely presented category. If 
\begin{displaymath}
\xymatrix{F'\ar[r]&F\ar[r]&F''}
\end{displaymath}
is an exact sequence in $C^\fp\Lmods$ then the induced sequence
\begin{displaymath}
\xymatrix{\overrightarrow{F'}\ar[r]&\overrightarrow{F}\ar[r]&\overrightarrow{F''}}
\end{displaymath}
is exact when evaluated at any object of $C$
\end{lem}
\begin{proof}
This follows directly from the \thref{extdef} and the fact that, for any object $c\in C$, the functor $C(-,c)|_{C^{\fp\op}}\in\Rmods C^\fp$ is flat \cite[Section 1.4]{crawleyboevey1994}.
\end{proof}
\subsection{Dualising a covariant representable functor}\label{isothms}
\begin{rmk}\thlabel{reflection}The functor $(-)^*:(K\Lmods)^\op\to K\Lmods$ preserves and reflects inverse limits. This can be seen by the fact that it is monadic (see \thref{monadic}), and all monadic functors preserve and reflect all limits (see \cite[VI.2, Exercises]{maclane1998} for this). In particular, this implies that a functor $F:R\Lmods\to\Ab$ preserves direct limits iff its dual $F^*:(R\Lmods)^\op\to\Ab$ preserves inverse limits.
\end{rmk}
\begin{lem}\thlabel{weirdiso}\cite[3.2.11]{enochsjenda2011} Let $M$ and $N$ be a left $R$-modules. There is a linear isomorphism
\begin{displaymath}
\sigma_{MN}:M^*\otimes N\to (N,M)^*
\end{displaymath}
such that $\sigma_{MN}(f\otimes n)(g)=f(g(n))$ for all $f\in M^*$ and $n\in N$, which is natural in $M$ and $N$. Also, $\sigma_{MN}$ is an isomorphism when $N$ is finitely presented.
\end{lem}
\begin{cor}\thlabel{nisomorphism}For any finitely presented left $R$-module $N$, $(N,-)^*\cong (-)^*\otimes N$ as functors $(R\Lmods)^\op\to\Ab$.
\end{cor}
\begin{cor}\thlabel{corweird}If a left $R$-module $N$ is finitely presented then the functor $(-)^*\otimes N:(R\Lmods)^\op\to\Ab$ preserves inverse limits and is finitely presented when restricted to $(R\lmods)^\op$.
\end{cor}
\begin{proof}Since $(-)^*\otimes N\cong (N,-)^*$, it follows from \thref{reflection} that it preserves inverse limits. Therefore we need only show that it is finitely presented when restricted to $(R\lmods)^\op$.

Choose a right exact sequence $R^m\to R^n\to N\to 0$ where $m$ and $n$ are positive integers. Since $\rmods(R\lmods)$ is closed under cokernels by \thref{covarabelian}, and by the fact that there is a right exact sequence $(-)^*\otimes R^m\to (-)^*\otimes R^n\to (-)^*\otimes N\to 0$, we need only show that $(-)^*\otimes R$ is finitely presented. Indeed, for each $M\in R\Lmods$, we have isomorphisms $M^*\otimes R\cong M^*\cong (R,M^*)\cong(M,R^*)$ which are natural in $M$ by \thref{homtens}, so $(-)^*\otimes R\cong (-,R^*)$, which is finitely presented since $R^*$ is a finitely presented left $R$-module.
\end{proof}
\subsection{Duality of finitely presented functors}\label{dualfun}
We have the usual duality of finitely presented modules $(-)^*:R\lmods\simeq(\rmods R)^\op$. This section looks at how it interacts with dualities of functors.

For any functor $F:R\lmods\to\Ab$, we define $dF:\rmods R\to\Ab$ to be the functor given by
\begin{displaymath}
(dF)M=(F,M\otimes -)
\end{displaymath}
for every $M\in\rmods R$. Thus we have a functor $d:((R\lmods)\Lmods)^\op\to(\rmods R)\Lmods$. This is discussed in \cite[10.3]{prest2009}. In particular, if $F:R\lmods\to\Ab$ is finitely presented, then so is $dF$. Furthermore, there is, for any $F:R\lmods\to\Ab$, a canonical natural transformation $F\to d^2F$ which is an isomorphism when $F$ is finitely presented. From this comes an equivalence $((R\lmods)\lmods)^\op\simeq(\rmods R)\lmods$.

\begin{prop}\thlabel{propaimpb}If $F:R\lmods\to\Ab$ is finitely presented then its dual $F^*:(R\lmods)^\op\to\Ab$ is also finitely presented.
\end{prop}
\begin{proof}
Find a morphism $f:A\to B$, with $A$ and $B$ finitely presented, such that $F$ is the cokernel of the induced map $(f,-):(B,-)\to(A,-)$. Then $F^*$ is the kernel of the induced map $(f,-)^*:(A,-)^*\to(B,-)^*$. But, by \thref{nisomorphism} and \thref{corweird}, $(B,-)^*$ and $(A,-)^*$ are finitely presented, and so, since $(R\lmods)\lmods$ is closed under kernels by \thref{covarabelian}, $F^*$ is finitely presented.
\end{proof}
\begin{rmk}\thlabel{aimpb2}\thref{propaimpb} and \thref{reflection} prove the \textbf{(a) implies (b)} part of \thref{contravarconditions}.
\end{rmk}
\begin{lem}\thlabel{fptens}\cite[6.1]{auslander1965} For any right $R$-module $N$, the the functor $N\otimes-:R\lmods\to\Ab$ is finitely presented iff $N$ is finitely presented.
\end{lem}
\begin{prop}
If $G:(R\lmods)^\op\to\Ab$ is finitely presented then its dual $G^*:R\lmods\to\Ab$ is also finitely presented.
\end{prop}
\begin{proof}Find a morphism $f:A\to B$, with $A$ and $B$ finitely presented, such that $G$ is the cokernel of the induced map $(-,f):(-,A)\to(-,B)$. Then $G^*$ is the kernel of the induced map $(-,f)^*:(-,B)^*\to(-,A)^*$. But, by \thref{weirdiso} and \thref{fptens}, $(-,B)^*\cong B^*\otimes-$ and $(-,A)^*\cong A^*\otimes -$ are finitely presented, and so, since $\rmods(R\lmods)$ is closed under kernels by \thref{noetheriancor}, $G^*$ is finitely presented.
\end{proof}
\begin{cor}\thlabel{equiv}For any finitely presented functor $F:R\lmods\to\Ab$, the canonical natural transformation $F\to F^{**}$ is an isomorphism.

For any finitely presented $G:(R\lmods)^\op\to\Ab$, the canonical natural transformation $G\to G^{**}$ is an isomorphism.

Therefore, there is an equivalence of categories $(-)^*:((R\lmods)\lmods)^\op\to\rmods(R\lmods)$ with pseudo-inverse $(-)^*:\rmods(R\lmods)\to ((R\lmods)\lmods)^\op$.
\end{cor}
\begin{proof}
Since $F$ is finitely presented, for any finitely generated left $R$-module $M$, $FM$ is a finitely generated $K$-module, and hence $FM\to(FM)^{**}$ is an isomorphism. The proof for $G$ is similar.
\end{proof}

\begin{rmk}The proof that the operation $d$ gives a duality of finitely presented functors (see \cite[10.3]{prest2009} for a proof) relies on the fact that the injectives in $(R\lmods)\lmods$ are precisely the functors of the form $M\otimes -$ for a finitely presented right $R$-module $M$. This is true for any ring $R$, but it is easier to prove if $R$ is an artin algebra since we can use \thref{homtens} and \thref{equiv}.
\end{rmk}

\begin{prop}\thlabel{eqdual}For any finitely presented functor $F:R\lmods\to\Ab$, there is an isomorphism $F^*\cong (dF)_*$ which is natural in $F$.
\end{prop}
\begin{proof}
Note that, for any functors $F:R\lmods\to\Ab$ and $G:(R\lmods)^\op\to \Ab$, there is an isomorphism $(F,G^*)\cong (G,F^*)$ which is natural in $F$ and $G$ by \thref{homtenspreadd}. It follows from this and \thref{weirdiso} that there are isomorphisms
\begin{displaymath}
(dF)_*(M)=(dF)(M^*)=(F,M^*\otimes -)\cong (F,(-,M)^*)\cong ((-,M),F^*)\cong F^*M
\end{displaymath}
which are natural in $F$ and $M$.
\end{proof}
\subsection{Extending contravariant functors along direct limits}\label{contraext}
Any functor $R\lmods\to\Ab$ can be uniquely extended to a direct limit preserving functor $R\Lmods\to\Ab$ by \thref{genunqext}. Here we give a contravariant analogue of this result.

For any functor $G:(R\lmods)^\op\to\Ab$, write $\overleftarrow{G}:(R\Lmods)^\op\to\Ab$ for the functor given by
\begin{displaymath}
\overleftarrow{G}M=((-,M),G)
\end{displaymath}
at any left $R$-module $M$ (the representable $(-,M)$ is to be read as a functor on $(R\lmods)^\op$). Also, for any natural transformation $\alpha:G\to G'$ in $\rmods(R\lmods)$, write $\overleftarrow\alpha:\overleftarrow{G}\to\overleftarrow{G'}$ for the induced natural transformation. By the Yoneda lemma, there is, for any $G\in\Rmods(R\lmods)$, an isomorphism $\overleftarrow{G}|_{(R\lmods)^\op}\cong G$ which is natural in $G$.
\begin{lem}\thlabel{xyz}For any additive functor $G:(R\lmods)^\op\to\Ab$, $\overleftarrow{G}:(R\Lmods)^\op\to\Ab$ preserves inverse limits.
\end{lem}
\begin{proof}
The Yoneda embedding $Y:R\Lmods\to ((R\lmods)^\op,\Ab)$ preserves direct limits. Therefore, the induced embedding $Y^\op:(R\Lmods)^\op\to ((R\lmods)^\op,\Ab)^\op$ preserves inverse limits. All representable functors preserve inverse limits. Therefore, since $\overleftarrow{G}$ is the composition
\begin{displaymath}
\xymatrix{(R\Lmods)^\op\ar[r]^{Y^\op}& ((R\lmods)^\op,\Ab)^\op\ar[rr]^{(-,G)}&&\Ab,}
\end{displaymath}
it does preserve inverse limits and, by the Yoneda lemma, it does restrict to $G$ on $(R\lmods)^\op$.
\end{proof}
\begin{lem}\thlabel{extofrest}For any functor $G:(R\Lmods)^\op\to\Ab$ which preserves inverse limits, and any left $R$-module $M$, there is an isomorphism
\begin{displaymath}
GM\cong \overleftarrow{G|_{(R\lmods)^\op}}M
\end{displaymath}
which is natural in $G$ and $M$.
\end{lem}
\begin{proof}
There is a homomorphism
\begin{displaymath}
\theta_M:GM\to ((-,M),G|_{(R\lmods)^\op})
\end{displaymath}
defined by sending an element $x\in GM$ to the natural transformation $\alpha_x:(-,M)\to G|_{(R\lmods)^\op}$ which, at any $L\in R\lmods$, has the component
\begin{displaymath}
\alpha_{xL}:(L,M)\to GL:f\mapsto G(f)(x).
\end{displaymath}
The map $\theta_M$ is natural in $M$, and, by the Yoneda lemma, it is an isomorphism when $M$ is finitely presented. Therefore, we have a natural transformation $\theta:G\to \overleftarrow{G|_{(R\lmods)^\op}}$ which is an isomorphism when evaluated at finitely presented left $R$-modules. However, every left $R$-module can be written as a direct limit (i.e. an inverse limit in $(R\Lmods)^\op$) of finitely presented left $R$-modules. Since both $G$ and $\overleftarrow{G|_{(R\lmods)^\op}}$ preserve inverse limits, it follows from \thref{lemmaf} that $\theta$ is an isomorphism.
\end{proof}
\begin{thm}\thlabel{unqext}Let $G:(R\lmods)^\op\to\Ab$ be any additive functor. Then $\overleftarrow{G}$ is the unique extension of $G$ to an additive functor $(R\Lmods)^\op\to\Ab$ which preserves inverse limits.
\end{thm}
\begin{proof}\thref{xyz} shows that $\overleftarrow{G}$ preserves inverse limits. We need only prove the uniqueness.

Let $G':(R\Lmods)^\op\to\Ab$ be any other such functor. Since $G'|_{(R\lmods)^\op}\cong G$, by \thref{extofrest}, we have $G'\cong \overleftarrow{G'|_{(R\lmods)^\op}}\cong \overleftarrow{G}$.
\end{proof}
\begin{prop}\thlabel{bimpa}Suppose that $G:(R\Lmods)^\op\to\Ab$ preserves inverse limits and is finitely presented when restricted to $(R\lmods)^\op$. Then $G\cong F^*$ for some functor $F:R\Lmods\to\Ab$ which preserves direct limits and is finitely presented when restricted to $R\lmods$.
\end{prop}
\begin{proof}
Let $F=\overrightarrow{G_0^*}$, the extension of $G^*_0$ to a functor $R\Lmods\to\Ab$ which preserves direct limits, where $G_0=G|_{(R\lmods)^\op}$. Then $F^*$ preserves inverse limits by \thref{reflection}, and $F^*|_{(R\lmods)^\op}=(F|_{R\lmods})^*\cong G_0^{**}\cong G_0$, so by \thref{unqext}, $F^*\cong G$.
\end{proof}
\begin{rmk}\thref{bimpa} shows that \textbf{(b) implies (a)} in \thref{contravarconditions}.
\end{rmk}
\begin{lem}\thlabel{dualhoms}For any finitely presented left $R$-module $A$, and any finitely presented right $R$-module $N$, there is an isomorphism $(N^*,A)\cong (A^*,N)$ which is natural in both $N$ and $A$.
\end{lem}
\begin{proof}
By repeated use of \thref{homtens} and the natural isomorphisms $A\cong A^{**}$ and $N\cong N^{**}$ (which are natural in $A$ and $N$ respectively), 
\begin{displaymath}
(N^*,A)\cong (N^*,A^{**})\cong (A^*,N^{**})\cong (A^*,N)
\end{displaymath}
which are natural in both $A$ and $N$.

Alternatively, one can use the fact that, for any functor $F:C\to D$, the maps $C(c,c')\to D(Fc,Fc')$ are natural in both variables. Therefore, by applying this to the equivalence $(-)^*:(R\lmods)^\op\to\rmods R$ and using the isomorphism $N\cong N^{**}$ (natural in $N$) we obtain isomorphisms
\begin{displaymath}
(N^*,A)\cong (A^*,N^{**})\cong (A^*,N)
\end{displaymath}
which are natural in both $A$ and $N$.
\end{proof}
\begin{lem}For any finitely presented functor $G:(R\lmods)^\op\to\Ab$, its predual $G_*$ is finitely presented.
\end{lem}
\begin{proof}
Find $(f:A\to B)\in R\lmods$ and a right exact sequence
\begin{displaymath}
\xymatrix{(-,A)\ar[r]^{(-,f)}&(-,B)\ar[r]&G\ar[r]&0}
\end{displaymath}
Then for any $N\in \rmods R$, by \thref{dualhoms}, we have a commutative diagram 
\begin{displaymath}
\xymatrix{(N^*,A)\ar[d]\ar[r]^{(N^*,f)}&(N^*,B)\ar[d]\ar[r]&G(N^*)\ar[r]&0\\
(A^*,N)\ar[r]^{(f^*,N)}&(B^*,N)}
\end{displaymath}
in which the vertical arrows are isomorphisms, both natural in $N$. Therefore $G_*$ is the cokernel of a morphism $((A^*,-)\to(B^*,-))\in(\rmods R,\Ab)$, and so since both $A^*$ and $B^*$ are finitely presented, $G_*$ is finitely presented.
\end{proof}
\begin{prop}\thlabel{extensionformula}For any finitely presented functor $G:(R\lmods)^\op\to\Ab$ and any left $R$-module $M$, there is an isomorphism
\begin{displaymath}
\overleftarrow GM\cong (G_*,-\otimes M)^*.
\end{displaymath}
which is natural in both $G$ and $M$. 
\end{prop}
\begin{proof}
Consider the functor $F:R\Lmods\to\Ab$, defined by $FM=(G_*,-\otimes M)$ for all $M\in R\Lmods$. Then $F$ is the composition 
\begin{displaymath}
\xymatrix{R\Lmods\ar[r]^T&(\rmods R)\Lmods\ar[rr]^{(G_*,-)}&&\Ab}
\end{displaymath}
where $T$ maps $M\in R\Lmods$ to $-\otimes M\in (\rmods R)\Lmods$. Both $T$ and $(G_*,-)$ preserve direct limits, and therefore $F$ preserves direct limits. Note that $F|_{R\lmods}=d(G_*)$.

Now let $G'=F^*$. By the proof of \thref{unqext}, there is an isomorphism $\overleftarrow{G'|_{(R\lmods)^\op}}\cong G'$ which is natural in $G'$, and hence in $G$. We need only show that there is an isomorphism $G'|_{(R\lmods)^\op}\cong G$ which is also natural in $G$. Indeed, by \thref{eqdual}, we have the isomorphisms
\begin{displaymath}
G'|_{(R\lmods)^\op}=(d(G_*))^*\cong (d(d(G_*)))_*\cong G_{**}\cong G,
\end{displaymath}
each of which is natural in $G$.
\end{proof}
\begin{cor}\thlabel{rightex}Let
\begin{displaymath}
\xymatrix{G'\ar[r]&G''\ar[r]&G\ar[r]&0}
\end{displaymath}
be a right exact sequence in $\rmods(R\lmods)$. The induced sequence
\begin{displaymath}
\xymatrix{\overleftarrow{ G'}\ar[r]&{\overleftarrow G''}\ar[r]&\overleftarrow G\ar[r]&0}
\end{displaymath}
is exact when evaluated at any left $R$-module
\end{cor}
\begin{proof}
The induced sequence
\begin{displaymath}
\xymatrix{G'_*\ar[r]&G''_*\ar[r]&G_*\ar[r]&0}
\end{displaymath}
is exact, and, for any left $R$-module $M$, so is
\begin{displaymath}
\xymatrix{0\ar[r]&(G_*,-\otimes M)\ar[r]&(G''_*,-\otimes M)\ar[r]&(G'_*,-\otimes M)}.
\end{displaymath}
But then, since dualising preserves exactness, the sequence
\begin{displaymath}
\xymatrix{(G'_*,-\otimes M)^*\ar[r]&(G''_*,-\otimes M)^*\ar[r]&(G_*,-\otimes M)^*\ar[r]&0}
\end{displaymath}
is exact. However, this sequence is isomorphic to 
\begin{displaymath}
\xymatrix{\overleftarrow{G'}M\ar[r]&\overleftarrow{G''}M\ar[r]&\overleftarrow GM\ar[r]&0}
\end{displaymath}
by \thref{extensionformula} (and naturally so in $M$).
\end{proof}
\begin{lem}\thlabel{littlething}For all left $R$-modules $M$ and $N$, there is an isomorphism $\overleftarrow{(-,N)}(M)\cong(M,N)$ which is natural in $M$ and $N$.
\end{lem}
\begin{proof}
By \thref{faithext} and the Yoneda embedding, the functor 
\begin{displaymath}
R\Lmods\to ((R\lmods)^\op,\Ab):A\mapsto (-,A)
\end{displaymath}
is fully faithful, and therefore there is an isomorphism
\begin{displaymath}
\overleftarrow{(-,N)}(M)=((-,M),(-,N))\cong(M,N)
\end{displaymath}
which is natural in $M$ and $N$.
\end{proof}
\begin{rmk}We now show how to prove that \textbf{(b) implies (c)} in \thref{contravarconditions}. Let $G:(R\Lmods)^\op\to\Ab$ be given, and assume that $G_0=G|_{(R\lmods)^\op}$ is finitely presented. Then there is a morphism $(f:A\to B)\in R\lmods$ and a right exact sequence
\begin{displaymath}
\xymatrix{(-,A)\ar[r]^{(-,f)}&(-,B)\ar[r]&G_0\ar[r]&0,}
\end{displaymath}
where the representable functors are read as functors on $(R\lmods)^\op$. By \thref{rightex}, there is an induced sequence
\begin{displaymath}
\xymatrix{\overleftarrow{(-,A)}\ar[r]^{\overleftarrow{(-,f)}}&\overleftarrow{(-,B)}\ar[r]&\overleftarrow{G_0}\ar[r]&0,}
\end{displaymath}
which is exact when evaluated at any left $R$-module. However, by \thref{littlething} $\overleftarrow{(-,f)}$ is isomorphic to $(-,f):(-,A)\to(-,B)$, which is now read as a natural transformation between functors on $(R\Lmods)^\op$, giving $\overleftarrow{G_0}\cong G$, so indeed we obtain a sequence as required by (c) in \thref{contravarconditions}. Clearly \textbf{(c) implies (d)}.
\end{rmk}
\begin{rmk}\thlabel{dimpb}To prove that \textbf{(d) implies (b)} in \thref{contravarconditions}, note that $\rmods(R\lmods)$ is closed under cokernels, so certainly $G|_{(R\lmods)^\op}$ is finitely presented. From \thref{rightex}, we see that $G$ is the extension of some finitely presented functor $(R\lmods)^\op\to\Ab$ to a functor which preserves  inverse limits.
\end{rmk}
\begin{rmk}\thlabel{cimpe}To prove that \textbf{(c) implies (e)} in \thref{contravarconditions}, take a sequence 
\begin{displaymath}
\xymatrix{(-,A)\ar[r]^{(-,f)}&(-,B)\ar[r]&G\ar[r]&0,}
\end{displaymath}
as described in (c). By the duality between finitely presented $R$-modules, there is an exact sequence 
\begin{displaymath}
\xymatrix{(-,A^{**})\ar[r]^{(-,f^{**})}&(-,B^{**})\ar[r]&G\ar[r]&0.}
\end{displaymath}
However, by \thref{homtens}, this gives an exact sequence 
\begin{displaymath}
\xymatrix{(A^{*},(-)^*)\ar[r]^{(f^{*},(-)^*)}&(B^{*},(-)^*)\ar[r]&G\ar[r]&0.}
\end{displaymath}
However, this shows that $G\cong F_*$, where $F$ is the cokernel of the natural transformation $(f^*,-):(A^*,-)\to(B^*,-)$ , thus establishing (e).

These steps are essentially reversible, so \textbf{(e) implies (c)} in \thref{contravarconditions}.
\end{rmk}
\subsection{Annihilator functors and duality of pp pairs}\label{annpp}

So far, the dualities we have considered have been purely algebraic. However, there is a duality between pp formulas which corresponds to the equivalence $d:((R\lmods)\lmods)^\op\to (\rmods R)\lmods$.

For a pp-$n$-formula $\varphi(\overline x)=\exists \overline wA\overline x\doteq B\overline w$ in the language of left $R$-modules, its \textbf{dual} is defined to be the pp-$n$-formula in the language of right $R$-modules given by
\begin{displaymath}
(D\varphi)(\overline x)=\exists\overline y\left(\overline x\doteq\overline yA\wedge\overline yB\doteq\overline 0\right).
\end{displaymath}
There is a similar map sending pp formulas in the language of left $R$-modules to pp formulas in the language of right $R$-modules. 

The operation $\varphi\mapsto D\varphi$ provides a duality on the level of pp pairs, since $\psi\leqslant\varphi$ implies $D\varphi\leqslant D\psi$ and $D^2\varphi=\varphi$ for all pp formulas $\varphi$ and $\psi$ in the language of left $R$-modules. In particular, to each pp pair $\varphi/\psi$ in the language of left $R$-modules, we have a corresponding pp pair $D\psi/D\varphi$ in the language of right $R$-modules. Thus, for a finitely presented functor $F_{\varphi/\psi}:R\lmods\to\Ab$ given by a pp pair $\varphi/\psi$, we have a corresponding functor $F_{D\psi/D\varphi}:\rmods R\to\Ab$. One can show, by using the short exact sequence resulting from \cite[10.3.6]{prest2009} and the snake lemma, that $dF_{\varphi/\psi}\cong F_{D\psi/D\varphi}$.

\begin{rmk}\thlabel{eeqf}The reason we consider this duality of pp formulas is the fact that $\mathbb{A}_\varphi(M)=\mathbb{F}_{D\varphi}(M^*)$ for any pp formula $\varphi$ in the language of left $R$-modules and any left $R$-module $M$ (see \cite[1.3.12]{prest2009}). Therefore
\begin{displaymath}
\mathbb{A}_{\varphi/\psi}=(\mathbb{F}_{D\psi/D\varphi})_*,
\end{displaymath}
which gives the next result immediately. It also proves that \textbf{(e) is equivalent to (f)} in \thref{contravarconditions}. \textbf{We have now proven \thref{contravarconditions}}.
\end{rmk}
\begin{cor}\thlabel{corollary}For any pp formula $\varphi$ in the language of left $R$-modules, the functor $\mathbb{A}_{\varphi/\psi}:(R\Lmods)^\op\to\Ab$ preserves inverse limits.
\end{cor}
\begin{proof}This follows directly from the observation that $\mathbb{A}_{\varphi/\psi}=(\mathbb{F}_{D\psi/D\varphi})_*$, from \thref{covarconditions} and \thref{contravarconditions}.
\end{proof}
Let $F:R\lmods\to\Ab$ be a finitely presented functor. On the surface, it seems like \thref{eqdual} showed that $F^*$ can be expressed as the more complicated $(dF)_*$. However, a lot of information about $(dF)_*$ can be gotten from the model theory of modules, so in fact we can now refer to some results from there to learn about $F^*$. 
\begin{lem}\thlabel{lemma}Let $F=F_{\varphi/\psi}:R\lmods\to\Ab$ be given by a pp pair $\varphi/\psi$ (see introduction). Then $F^*_{\varphi/\psi}\cong A_{\varphi/\psi}$.
\end{lem}
\begin{proof}
By \thref{eqdual}, it is enough to show that $(dF_{\psi/\varphi})_*\cong A_{\varphi/\psi}$. First, note that $A_\varphi(M)=F_{D\varphi}(M^*)$ for any finitely presented left left $R$-module $M$ and that the same holds for $\psi$. There are, by \cite[10.3.6]{prest2009}, short exact sequences
\begin{displaymath}
\xymatrix{0\ar[r]&\ar[r]F_{D\varphi}&(R^n,-)\ar[r]&dF_\varphi\ar[r]&0\\0\ar[r]&\ar[r]F_{D\psi}&(R^n,-)\ar[r]&dF_\psi\ar[r]&0.}
\end{displaymath}
Combining these results gives us a commutative diagram
\begin{displaymath}
\xymatrix{&0\ar[d]&0\ar[d]&0\ar[d]\\0\ar[r]&0\ar[d]\ar[r]&A_\psi\ar[r]\ar[d]&A_\varphi\ar[d] &\\
0\ar[r]&0\ar[r]\ar[d]&(R^n,(-)^*)\ar[d]\ar[r]&(R^n,(-)^*)\ar[d]\ar[r]&0\\
0\ar[r]&(dF_{\varphi/\psi})_*\ar[d]\ar[r]& (dF_\varphi)_*\ar[d]\ar[r]& (dF_\psi)_*\ar[d]\ar[r]&0\\
&(dF_{\varphi/\psi})^*\ar[d]\ar[r]&0\ar[d]\ar[r]&0\ar[d]\ar[r]&0
\\&0&0&0}
\end{displaymath}
with exact rows, and the snake lemma gives us a short exact sequence
\begin{displaymath}
\xymatrix{0\ar[r]&A_\varphi\ar[r]&A_\psi\ar[r]&(dF_{\varphi/\psi})^*\ar[r]&0}
\end{displaymath}
so $(dF_{\varphi/\psi})^*\cong A_\psi/A_\varphi$, as required.
\end{proof}
\begin{cor}\thlabel{meh}For any pp pair $\varphi/\psi$, $\mathbb{F}_{\varphi/\psi}^*\cong\mathbb{A}_{\varphi/\psi}$.
\end{cor}
\begin{proof}
It follows from the equivalence of (b) and (f) in \thref{contravarconditions} that $\mathbb{A}_{\varphi/\psi}$ preserves inverse limits. Since \thref{lemma} shows that $\mathbb{F}^*_{\varphi/\psi}$ and $\mathbb{A}_{\varphi/\psi}$ agree on $(R\lmods)^\op$, \thref{unqext} shows that they agree on $(R\Lmods)^\op$.
\end{proof}
\begin{rmk}One can show directly, without using (f) in \thref{contravarconditions}, that a functor of the form $\mathbb{A}_{\varphi}$ for a pp formula $\varphi$ preserves inverse limits. This gives an alternative proof that the annihilator functor $\mathbb{A}_{\varphi/\psi}$ of a pp pair $\varphi/\psi$ preserves inverse limits, which goes as follows. We know that there is a sequence
\begin{displaymath}
\xymatrix{\mathbb{A}_\varphi\ar[r]&\mathbb{A}_\psi\ar[r]&\mathbb{A}_{\varphi/\psi}\ar[r]&0}
\end{displaymath}
which is exact when evaluated at any left $R$-module. It follows from the equivalence of (b) and (d) in \thref{contravarconditions} that $\mathbb{A}_{\varphi/\psi}$ preserves inverse limits.
\end{rmk}
\section{Dualising locally finitely presented categories}
Here we extend some of our results so that they apply to contravariant functors on more general categories than just $R\Lmods$ for an artin algebra $R$. Let $A$ denote a small, additive, $K$-linear category. Despite them not being precisely the same thing (though, they are isomorphic), we write $(A,K\lmods)$ for the category of functors $A\to\Ab$ for which $Fa$ is a finitely generated $K$-module for every object $a\in A$, and similarly for $(A^\op,K\lmods)$.  Clearly we have equivalences
\begin{align*}
(A,K\lmods)^\op\to(A^\op,K\lmods)&:F\mapsto F^*\\
(A^\op,K\lmods)^\op\to(A,K\lmods)&:G\mapsto G^*.
\end{align*}
\begin{defn}\cite{ar1974} A small, additive, $K$-linear category $A$ is called a \textbf{dualising variety} (or \textbf{dualising $K$-variety}) if it satisfies each of the following:
\begin{list}{*}{â€¢}
\item[(a)]$A$ is \textbf{finite}, which means that, for any objects $a,b\in A$, the hom-set $A(a,b)$ is a finitely generated $K$-module.
\item[(b)]It is a \textbf{variety}, which means that all idempotents split in $A$ (it is also said that $A$ is \textbf{Karoubian} or \textbf{idempotent complete} when this is the case).
\item[(c)]The equivalence \begin{displaymath}
(-)^*:(A^\op,K\lmods)^\op\to(A,K\lmods)
\end{displaymath} restricts to an equivalence $(\rmods A)^\op\to A\lmods$ .
\item[(d)]The equivalence 
\begin{displaymath}
(-)^*:(A,K\lmods)^\op\to(A^\op,K\lmods)
\end{displaymath}
restricts to an equivalence $(A\lmods)^\op\to \rmods A$ .
\end{list}
\end{defn}
\begin{rmk} In order for the definition of dualising variety to make sense, note that if $A$ is finite then $A\lmods$ is a subcategory of $(A,K\lmods)$ and $\rmods A$ is a subcategory of $(A^\op,K\lmods)$.
\end{rmk}
\begin{rmk}If $A$ is a dualising variety then so is $A^\op$.
\end{rmk}
\begin{thm}\cite[2.4]{ar1974}\thlabel{varcrit} A finite variety $A$ is a dualising variety if and only if it satisfies each of the following conditions.
\begin{list}{*}{}
\item[(a)]$A$ and $A^\op$ have weak cokernels.
\item[(b)]For every object $b\in A$ there is an object $c\in A$ such that
\begin{displaymath}
A(a,b)\to \text{Hom}_{A(c,c)^\op}(A(c,a),A(c,b))
\end{displaymath}
is an isomorphism for every object $a\in A$.
\item[(c)]For every object $a\in A$ there is an object $d\in A$ such that
\begin{displaymath}
A(a,b)\to \text{Hom}_{A(d,d)}(A(b,d),A(a,d))
\end{displaymath}
is an isomorphism for every object $b\in A$.
\end{list}
\end{thm}
\begin{cor}If $A$ is a dualising variety then $A\lmods$ is an abelian subcategory of $A\Lmods$ and $\rmods A$ is an abelian subcategory of $\Rmods A$.
\end{cor}
\begin{proof}This follows immediately from \thref{varcrit} and \thref{weakcokernels}.
\end{proof}
\begin{ex}It is clear from Section \ref{dualfun} that if $R$ is an artin algebra then $R\lmods$ is a dualising variety.
\end{ex}
\begin{ex}\thlabel{proj}It follows from \thref{varcrit} that, for any artin algebra $R$, $R\text{-}\textbf{proj}$, the category of finitely generated projective left $R$-modules, is a dualising variety (see \cite[2.5]{ar1974}).
\end{ex}
\begin{thm}\cite[2.6]{ar1974} If $A$ is a dualising variety then so is $A\lmods$.
\end{thm}
Now we go on to see how these concepts give a strengthening of the main results.
\begin{defn}We say that an additive $K$-linear category $C$ is a \textbf{dlfp $K$-linear category} (or just \textbf{dlfp category}) if $C$ is locally finitely presented and $C^\text{fp}$ is a dualising $K$-variety.
\end{defn}
\begin{rmk}Here dlfp is an abbreviation for ``dualising locally finitely presented''.
\end{rmk}
\begin{cor}If $C$ is a dlfp category then so are $C^\text{fp}\Lmods$ and $\Rmods C^\text{fp}$.
\end{cor}
\begin{thm}Suppose $C$ is a locally finitely presented category. Then for every additive functor $G:C^\text{fpop}\to\Ab$, there is a unique functor $\overleftarrow{G}:C^\op\to\Ab$ which preserves inverse limits such that $\overleftarrow{G}|_{C^\text{fpop}}\cong G$. In particular, $\overleftarrow{G}$ is given by
\begin{displaymath}
\overleftarrow{G}c=(C(-,c)|_{C^\text{fpop}},G).
\end{displaymath}
\end{thm}
\begin{proof}
This is a straightforward generalisation of \thref{unqext}.
\end{proof}
\begin{thm}\thlabel{woosh}Let $C$ be a $K$-linear locally finitely presented category. For any additive functor $F:C^\fp\to\Ab$ and any object $c\in C$, there is an isomorphism
\begin{displaymath}
\left(\overrightarrow{F}c\right)^*\cong \overleftarrow{F^*}c
\end{displaymath}
which is natural in $F$ and $c$.
\end{thm}
\begin{proof}
By \thref{tensyoneda}, there are isomorphisms
\begin{align*}
\left(\overrightarrow{F}c\right)^*&=(C(-,c)|_{C^{\fp\op}}\otimes_{C^\fp}F)^*
\\&\cong (C(-,c)|_{C^{\fp\op}},F^*)
\\&=\overleftarrow{F^*}c
\end{align*}
which are natural in $F$ and $c$.
\end{proof}
\begin{cor}Let $C$ be a dlfp category. For any finitely presented functor $G:C^\text{fpop}\to\Ab$ and any object $c\in C$, there is an isomorphism
\begin{displaymath}
\overleftarrow{G}c\cong \left(\overrightarrow{G^*}c\right)^*
\end{displaymath}
natural in both $c$ and $G$.
\end{cor}
\begin{proof}
Since $C^\fp$ is a dualising variety and $G$ is finitely presented, there is an isomorphism $G^{**}\cong G$ which is natural in $G$.

By \thref{woosh} there are isomorphism
\begin{align*}
\left(\overrightarrow{G^*}c\right)^*\cong\overleftarrow{G^{**}}c\cong \overleftarrow{G}c
\end{align*}
which are natural in both $c$ and $G$.
\end{proof}
\begin{cor}[Generalisation of \thref{rightex}]\thlabel{rightexactgen}If $C$ is a dlfp category, $G,G',G'':C^\text{fpop}\to\Ab$ are finitely presented additive functors and 
\begin{displaymath}
G'\to G''\to G\to 0
\end{displaymath}
is an exact sequence, then the induced sequence
\begin{displaymath}
\overleftarrow{G'}\to \overleftarrow{G''}\to \overleftarrow{G}\to 0
\end{displaymath}
is exact when evaluated at any object in $C$.
\end{cor}
\begin{proof}
Follows directly from \thref{extex} and \thref{woosh}.
\end{proof}
We now have the following generalisation of (some parts of) \thref{contravarconditions}, which follows by a similar argument.
\begin{thm}\thlabel{genconcon}Let $C$ be a dlfp category and let $G:C^\op\to\Ab$ be any additive functor. The following are equivalent:
\begin{list}{*}{}
\item[(a)]$G\cong F^*$ for some functor $F:C\to\Ab$ which preserves direct limits and is finitely presented when restricted to $C^\text{fp}$.
\item[(b)]$G$ preserves inverse limits and is finitely presented when restricted to $C^\text{fpop}$.
\item[(c)]There are objects $a,b\in C^\text{fp}$ and a sequence of natural transformations
\begin{displaymath}
\xymatrix{C(-,a)\to C(-,b)\to G\to 0}
\end{displaymath}
which is exact when evaluated at any object.
\item[(d)]There is a sequence of natural transformations 
\begin{displaymath}
G'\to G''\to G\to 0
\end{displaymath}
which is exact when evaluated at any object, where the functors $G',G'':C\to\Ab$ preserve inverse limits and are finitely presented when restricted to $C^\text{fpop}$.
\end{list}
\end{thm}
\begin{cor}Let $C$ be a dlfp. Suppose $G:C^\op\to\Ab$ is an additive functor such that $G|_{C^{\fp\op}}$ is finitely presented. If $G|_{C^{\fp\op}}:C^{\fp\op}\to\Ab$ is faithful then $G$ is faithful.  

If $a\in C^{\fp}$ is an cogenerator for $C^\fp$ then it is also a cogenerator for $C$.
\end{cor}
\begin{proof}
By \thref{genconcon}, $G\cong F^*$ for a functor $F:C\to\Ab$ which preserves direct limits and is finitely presented when restricted to $C^\fp$. We can view $F$ as a functor $C\to K\Lmods$. Since $F|_{C^\fp}$ is finitely presented and $C^\fp$ is a finite variety, it follows from \thref{fpfunc1} that $Fc$ is a finitely presented $K$-module.

The functor $F|_{C^\fp}$ is faithful, since for any morphism $f$ in $C^\fp$,
\begin{displaymath}
Ff=0\Rightarrow (Ff)^*\Rightarrow Gf=0\Rightarrow f=0.
\end{displaymath}
It follows from \thref{faithext} that $F$ is faithful. Therefore, for any morphism $g$ in $C$,
\begin{displaymath}
Gg=0\Rightarrow (Fg)^*=0\Rightarrow Fg=0\Rightarrow g=0.
\end{displaymath}
Therefore $G$ is faithful. 

The second statement follows from the first by considering $G=C(-,a)|_{C^\fp}$.
\end{proof}
The ``existence'' statement in \thref{genconcon}(a) can be strengthened to an ``existence and uniqueness'' statement by the following result.

\begin{prop}Let $C$ be a dlfp. Suppose $F_1:C\to\Ab$ is an additive functor which preserves direct limits and restricts to a finitely presented functor on $C^\fp$. If an additive functor $F_2:C\to\Ab$ is finitely presented when restricted to $C^\fp$ then
\begin{displaymath}
F_1^*\cong F_2^*\Rightarrow F_1\cong F_2.
\end{displaymath}
\end{prop}
\begin{proof}By \thref{reflection}, $F_1^*:C^\op\to\Ab$ preserves inverse limits, so again by \thref{reflection}, $F_2$ preserves direct limits. Therefore, $F_1|_{C^\fp}\cong F_1|_{C^\fp}^{**}\cong F_2|_{C^\fp}^{**}\cong F_2|_{C^\fp}$, and so, by \thref{genunqext}, $F_1\cong F_2$.
\end{proof}
\begin{ex}\thlabel{examp}In \cite[1.4]{crawleyboevey1994}, a one-to-one correspondence is established between small additive categories with split idempotents and locally finitely presented additive categories. This correspondence sends a small additive category $A$ with split idempotents to the category $\textbf{Flat}\text{-}A$ of all flat objects in $\Rmods A$, and sends a locally finitely presented additive category to its category of finitely presented objects. The category of finitely presented objects in $\textbf{Flat}\text{-}A$ is the category of all finitely presented projectives in $\Rmods A$. Therefore, for any module category, the category of flat modules is locally finitely presented, and the finitely presented projective modules are the finitely presented objects in the category of flat objects.

It follows from these facts and \thref{proj} that $\textbf{Flat}\text{-}R$ is a dlfp category when $R$ is an artin algebra.
\end{ex}
\begin{rmk}The one-to-one correspondence mentioned in \thref{examp} induces a one-to-one correspondence between dualising varieties and dlfp categories (in the sense that two $K$-linear categories should be considered to be ``the same'' iff there is a $K$-linear equivalence between them).
\end{rmk}

\end{document}